\documentclass[12pt]{article}

\usepackage{amssymb,amsthm,amsmath}
\usepackage{mathrsfs}
\usepackage[dvipdfmx]{graphicx}
\usepackage[FIGTOPCAP,nooneline]{subfigure}
\usepackage{geometry}

\usepackage{color}

\makeatletter
    
    \@addtoreset{equation}{section}
\makeatother

\theoremstyle{definition}
\newtheorem{thm}{Theorem}[section]

\newtheorem{prop}[thm]{Proposition}
\newtheorem{lem}[thm]{Lemma}
\newtheorem{rem}[thm]{Remark}

\usepackage[T1]{fontenc}
\usepackage[utf8]{inputenc}
\usepackage{authblk}

\title{Global solvability and convergence of the Euler-Poincar\'{e} regularization of the two-dimensional Euler equations}

\author{Takeshi Gotoda \footnote{ Department of mathematics, Kyoto University, Kitashirakawa Oiwake-cho, Sakyo-ku, Kyoto, JAPAN  email : gotoda@math.kyoto-u.ac.jp } }

\date{}

\begin{document}


\maketitle

\begin{abstract}
We study the Euler-Poincar\'{e} equations that are the regularized Euler equations derived from the Euler-Poincar\'{e} framework. It is noteworthy to remark that the Euler-Poincar\'{e} equations are a generalization of two well-known regularizations, the vortex blob method and the Euler-$\alpha$ equations. We show the global existence of a unique weak solution for the two-dimensional (2D) Euler-Poincar\'{e} equations with the initial vorticity in the space of Radon measure. This is a remarkable feature of these equations since the existence of weak solutions with the Radon measure initial vorticity has not been established in general for the 2D Euler equations. We also show that weak solutions of the 2D Euler-Poincar\'{e} equations converge to those of the 2D Euler equations in the limit of the regularization parameter when the initial vorticity belongs to the space of integrable and bounded functions.

\end{abstract}

\section{Introduction}

The motion of inviscid incompressible 2D flows is described by the 2D Euler equations,
\begin{equation}
\partial_t \boldsymbol{u} + (\boldsymbol{u} \cdot \nabla) \boldsymbol{u} + \nabla p = 0, \qquad \operatorname{div} \boldsymbol{u} = 0, \qquad \boldsymbol{u}(\boldsymbol{x}, 0) = \boldsymbol{u}_0(\boldsymbol{x}),  \label{Euler}
\end{equation}
where $\boldsymbol{u}=\boldsymbol{u}(\boldsymbol{x}, t)$ is the fluid velocity field and $p=p(\boldsymbol{x}, t)$ is the scalar pressure. Taking the curl of (\ref{Euler}) and defining the vorticity field by $\omega = \operatorname{curl} \boldsymbol{u}$, we obtain the transport equation for $\omega$, 
\begin{equation}
\partial_t \omega + (\boldsymbol{u} \cdot \nabla) \omega = 0, \qquad \omega(\boldsymbol{x}, 0) = \omega_0(\boldsymbol{x}),  \label{vEuler}
\end{equation}
and $\omega_0 = \operatorname{curl} \boldsymbol{u}_0$. Then, the velocity $\boldsymbol{u}$ is recovered from the vorticity $\omega$ via the Biot-Savart law,
\begin{equation}
\boldsymbol{u}(\boldsymbol{x}) = \left( \boldsymbol{K} \ast \omega \right)(\boldsymbol{x}) = \int_{\mathbb{R}^2} \boldsymbol{K}(\boldsymbol{x} - \boldsymbol{y}) \omega(\boldsymbol{y}) d \boldsymbol{y}, \label{biot}
\end{equation}
in which $\boldsymbol{K}$ denotes a singular integral kernel defined by 
\begin{equation*}
\boldsymbol{K}(\boldsymbol{x}) = \nabla^\perp G(\boldsymbol{x}) = - \frac{1}{2 \pi} \frac{\boldsymbol{x}^\perp}{\left| \boldsymbol{x} \right|^2}, \qquad G(\boldsymbol{x}) = - \frac{1}{2 \pi} \log{|\boldsymbol{x}|},
\end{equation*}
with $\nabla^\perp = (\partial_{x_2}, - \partial_{x_1})$ and $\boldsymbol{x}^\perp = (x_2, - x_1)$. Here, the function $G$ is the fundamental solution to the 2D Laplacian. If a passive particle is initially placed at the position $\boldsymbol{x}$, its evolution is governed by the following equation,
\begin{equation}
\partial_t \boldsymbol{\eta}(\boldsymbol{x}, t) = \boldsymbol{u} \left( \boldsymbol{\eta}(\boldsymbol{x}, t), t \right), \qquad \boldsymbol{\eta}(\boldsymbol{x}, 0) = \boldsymbol{x}, \label{eta}
\end{equation}
where $\boldsymbol{\eta}(\boldsymbol{x}, t)$ is called the Lagrangian flow map describing the position of the passive point $\boldsymbol{x}$ at time $t$. The solution of (\ref{eta}) yields that of (\ref{vEuler}) with the initial vorticity $\omega_0$ since it follows that
\begin{equation}
\omega(\boldsymbol{x}, t) = \omega_0 \left( \boldsymbol{\eta}(\boldsymbol{x}, - t) \right). \label{v-def}
\end{equation}
Hence, finding a solution to (\ref{eta}) with (\ref{biot}) and (\ref{v-def}) is equivalent to solving the problem (\ref{vEuler}) with (\ref{biot}). Regarding the initial value problem of the 2D Euler equations, the global existence of a unique weak solution has been established for $\omega_0 \in L^1(\mathbb{R}^2) \cap L^\infty(\mathbb{R}^2)$ \cite{Marchioro, Yudovich}. We also have global weak solutions with $\omega_0 \in L^1(\mathbb{R}^2) \cap L^p(\mathbb{R}^2)$ and $1 < p < \infty$, though the uniqueness is still open \cite{Diperna}. Moreover, the existence theorem can be extended to the case that $\omega_0$ belongs to the space of finite Radon measure on $\mathbb{R}^2$, denoted by $\mathcal{M}(\mathbb{R}^2)$, with a distinguished sign and its induced velocity $\boldsymbol{u}_0 \in L_{loc}^2(\mathbb{R}^2)$ \cite{Delort, Evans, Majda}. Note that this class includes vortex sheets but not point vortices.

In order to consider the vortex dynamics modeled by singular vortices in the space of Radon measure, such as vortex sheets or point vortices, it is effective to analyze the evolution of them in the regularized Euler flow, since the global existence of a unique solution is guaranteed. Actually, the Euler-$\alpha$ equations, which is a dispersive regularization of the Euler equations, have a unique global weak solution for initial vorticity in $\mathcal{M}(\mathbb{R}^2)$ \cite{Oliver} and point vortex solutions of them are used to analyze a singular vortex dynamics in ideal incompressible flows \cite{G.}. Another important example of the regularization is the vortex blob method that is introduced to compute the evolution of vortex sheets \cite{Chorin, Krasny}. Although the Euler-$\alpha$ equations and the vortex blob method are originally derived in different backgrounds and are usually applied to the different problems, we can derive both of them based on the Euler-Poincar\'{e} framework \cite{Holm(c)}. In the present paper, we investigate the Euler-Poincar\'{e} equations that come through the Euler-Poincar\'{e} framework and a generalization of both the Euler-$\alpha$ equations and the vortex blob method. We prove the existence of a unique global weak solution with its vorticity in the space of Radon measure. Moreover, we show that weak solutions of the Euler-Poincar\'{e} equations converge to those of the Euler equations in the limit of the regularization parameter provided the initial vorticity is integrable and bounded.

This paper is organized as follows. Section~\ref{general} presents the general framework of a regularization of the Euler equations based on the theory of the Euler-Poincar\'{e} equations. Especially, we observe two examples of the regularization, the vortex blob method and the Euler-$\alpha$ equations, in the viewpoint of the Euler-Poincar\'{e} framework in Section~\ref{blob} and \ref{alpha} respectively. In Section~\ref{Main}, we state the main results consisting of the global solvability of the Euler-Poincar\'{e} equations in the space of Radon measure and the convergence of weak solutions to those in the Euler equations for the bounded initial vorticity. The proof of the solvability and the convergence are shown in Section~\ref{proof1} and \ref{proof2} respectively.

\section{The Euler-Poincar\'{e} regularization}
\subsection{General framework}
In this section, we derive a regularized 2D Euler equations based on the framework given in \cite{Foias, Holm(c)}. Let $\boldsymbol{v}$ be an incompressible velocity field that we call {\it singular velocity}. We define the {\it regularized  velocity} by 
\begin{equation}
\boldsymbol{u}_h(\boldsymbol{x}) = \left( h \ast \boldsymbol{v} \right) (\boldsymbol{x}) = \int_{\mathbb{R}^2} h\left( \boldsymbol{x} - \boldsymbol{y} \right) \boldsymbol{v}(\boldsymbol{y}) d\boldsymbol{y}, \label{Rvelo}
\end{equation}
where $h$ is a scalar valued function on $\mathbb{R}^2$. In this paper, we assume that $h$ is a integrable function and may have a singularity at the origin. Owing to the convolution with $h$, $\boldsymbol{u}$ is smoother than $\boldsymbol{v}$. Similarly, we define the {\it singular vorticity} $q$ and the {\it regularized vorticity} $\omega_h$ by
\begin{equation*}
q = \operatorname{curl} \boldsymbol{v}, \qquad \omega_h = \operatorname{curl} \boldsymbol{u}_h.
\end{equation*}
Note that it follows that $\operatorname{div} \boldsymbol{u} = 0$ and $\omega_h = h \ast q$ when the convolution commutes with the differential operator.

Now, we consider the Hamiltonian structure with the Hamiltonian,
\begin{equation*}
\mathscr{H} = \frac{1}{2} \int_{\mathbb{R}^2} \boldsymbol{v}(\boldsymbol{x}) \cdot \boldsymbol{u}_h(\boldsymbol{x}) d\boldsymbol{x},
\end{equation*}
and the Lagrangian flow map $\boldsymbol{\eta}_h$ induced by regularized velocity,
\begin{equation}
\partial_t \boldsymbol{\eta}_h(\boldsymbol{x}, t) = \boldsymbol{u}_h \left( \boldsymbol{\eta}_h(\boldsymbol{x}, t), t \right), \qquad  \boldsymbol{\eta}_h(\boldsymbol{x}, 0) = \boldsymbol{x}. \label{h-eta}
\end{equation}
Then, the Euler-Poincar\'{e} equations arise from an application of Hamilton's principle as follows.
\begin{equation}
\partial_t \boldsymbol{v} + (\boldsymbol{u}_h \cdot \nabla) \boldsymbol{v} - (\nabla \boldsymbol{v})^T \cdot \boldsymbol{u}_h - \nabla \Pi = 0, \label{REE}
\end{equation}
where $\Pi$ is a generalized pressure. Taking the curl of (\ref{REE}) with the divergence-free condition, we obtain the transport equation for the singular vorticity convected by the regularized velocity,
\begin{equation*}
\partial_t q + (\boldsymbol{u}_h \cdot \nabla) q = 0.
\end{equation*}
Note that the Biot-Savart law gives $\boldsymbol{u}_h = \boldsymbol{K} \ast \omega_h$. Then, owing to the relation $\omega_h = h \ast q$, we finally obtain the vorticity form of the Euler-Poincar\'{e} equations.
\begin{equation}
\partial_t q + (\boldsymbol{u}_h \cdot \nabla) q = 0, \qquad \boldsymbol{u}_h = \boldsymbol{K}_h \ast q, \qquad \boldsymbol{K}_h = \boldsymbol{K} \ast h. \label{RPDE}
\end{equation}
It is important to remark that the regularized kernel $\boldsymbol{K}_h$ satisfies 
\begin{equation}
\boldsymbol{K}_h = \nabla^\perp G_h, \qquad - \Delta G_h = h, \label{KGh}
\end{equation}
where $G_h$ is a solution for the Poisson equation and thus we have $G_h = G \ast h$. Considering the definition of the singular integral kernel, we can see that a feature of the regularized equations (\ref{RPDE}) comes from how to regularize the delta function by the smoothing function $h$.

\label{general}

\subsubsection{The vortex blob method}
The vortex blob method is one of the most common regularizations of the Euler flow. The method was introduced to compute the vortex sheets' evolutions numerically\cite{Anderson, Chorin, Krasny} and has been applied to a variety of fluid dynamics, for instance see \cite{Leonard}. In this method, the regularized integral kernel $\boldsymbol{K}_\delta$ is given by
\begin{equation*}
\boldsymbol{K}^\delta(\boldsymbol{x}) = - \frac{1}{2 \pi} \frac{\boldsymbol{x}^\perp}{|\boldsymbol{x}|^2 + \delta^2}. 
\end{equation*}
This vortex blob regularization can be derived from the Euler-Poincar\'{e} framework. Indeed, replacing the function $h$ in (\ref{Rvelo}) by
\begin{equation*}
h^\delta(\boldsymbol{x}) = \frac{1}{2 \pi \delta^2} \psi\left( \frac{|\boldsymbol{x}|}{\delta} \right), \qquad \psi(r) = \frac{2}{(r^2 + 1)^2},
\end{equation*} 
we find that $\boldsymbol{K}^\delta$ comes through the solution for the following Poisson equation.
\begin{equation*}
\boldsymbol{K}^\delta = \nabla^\perp G^\delta, \qquad - \Delta G^\delta = h^\delta.
\end{equation*}
It is easily confirmed that the solution $G^\delta$ is given by
\begin{equation*}
G^\delta(\boldsymbol{x}) = G^\delta_r(|\boldsymbol{x}|) = - \frac{1}{2 \pi} \log{\sqrt{|\boldsymbol{x}|^2 + \delta^2}},
\end{equation*}
and $\boldsymbol{K}^\delta$ is rewritten by
\begin{equation*}
\boldsymbol{K}^\delta(\boldsymbol{x}) = \boldsymbol{K}(\boldsymbol{x})\Psi\left( \frac{|\boldsymbol{x}|}{\delta} \right), \qquad \Psi(r) = - 2 \pi r \frac{dG^{\delta=1}_r}{dr}(r) = \frac{r^2}{r^2 + 1}.
\end{equation*}
Note that we have $\Psi'(r) = r \psi(r) $. Hence, both $G^\delta$ and $\boldsymbol{K}^\delta$ have no singularity, which is the essence of the regularization. A remarkable feature of the vortex blob regularization is that the smoothing function $h^\delta$ is bounded and decays algebraically. Regarding the relation with the Euler equations in the $\delta \rightarrow 0$ limit, it has been shown in \cite{Liu} that the vortex blob methods converge to weak solutions of the 2D Euler equations under the conditions that the vorticity is in $\mathcal{M}(\mathbb{R}^2)$ with a distinguished sign and its induced velocity is locally square integrable.

\label{blob}

\subsubsection{The Euler-$\alpha$ model}
The Euler-$\alpha$ model was first introduced in \cite{Holm(a), Holm(b)} as the higher dimensional Camassa-Holm equations. We can also derive it by applying Lagrangian averaging to the Euler equations \cite{Marsden}. A remarkable property of this model is that the regularized velocity $\boldsymbol{u}^\alpha$ is written by $\boldsymbol{u}^\alpha = (1 - \alpha^2 \Delta)^{-1} \boldsymbol{v}$, see \cite{Foias}, and thus the smoothing function $h^\alpha$ in (\ref{Rvelo}) is defined by a fundamental solution for the operator $1 - \alpha^2 \Delta$, namely 
\begin{equation*}
h^\alpha(\boldsymbol{x}) = \frac{1}{2\pi\alpha^2} K_0 \left( \frac{|\boldsymbol{x}|}{\alpha} \right),
\end{equation*}
in which $K_0$ denotes the modified Bessel function of the second kind. It is important to remark that $K_0$ has a singularity at the origin like $K_0(r) \sim -\log{r}$ as $r \rightarrow 0$ and decays exponentially. Since the singular velocity $\boldsymbol{v}$ is explicitly expressed by the regularized velocity $\boldsymbol{u}^\alpha$, we find the Euler-Poincar\'{e} equations for $\boldsymbol{u}^\alpha$ as follows.
\begin{equation*}
(1 - \alpha^2 \Delta) \partial_t \boldsymbol{u}^\alpha + \boldsymbol{u}^\alpha \cdot \nabla (1 - \alpha^2 \Delta) \boldsymbol{u}^\alpha + (\nabla \boldsymbol{u}^\alpha)^T \cdot (1 - \alpha^2 \Delta) \boldsymbol{u}^\alpha = \nabla \Pi.
\end{equation*}
This equations are known as the Euler-$\alpha$ equations. Regarding the evolution of the singular vorticity described by the equation corresponding to (\ref{RPDE}), there exist a unique global weak solution for the initial vorticity $q_0 \in \mathcal{M}(\mathbb{R}^2)$ in contrast to the 2D Euler equations \cite{Oliver}. Moreover, it is also shown that a weak solution for $q_0 \in L^\infty(\mathbb{R}^2)\cap L^1(\mathbb{R}^2)$ converges to that of the vorticity equations induced by the 2D Euler equations. In terms of the formulation in (\ref{KGh}), we have
\begin{equation*}
\boldsymbol{K}^\alpha = \nabla^\perp G^\alpha, \qquad - \Delta G^\alpha(\boldsymbol{x}) = h^\alpha(\boldsymbol{x}),
\end{equation*}
and the solution $G^\alpha$ for the Poisson equation is concretely given by
\begin{equation*}
G^\alpha(\boldsymbol{x}) = G^\alpha_r(|\boldsymbol{x}|) = - \frac{1}{2 \pi} \left[ \log{|\boldsymbol{x}|} + K_0\left( \frac{|\boldsymbol{x}|}{\alpha} \right) \right].
\end{equation*}
Hence, we obtain the regularized integral kernel $\boldsymbol{K}^\alpha$ as follows.
\begin{equation*}
\boldsymbol{K}^\alpha(\boldsymbol{x}) = - \frac{1}{2 \pi} \frac{\boldsymbol{x}^\perp}{|\boldsymbol{x}|^2} B_K \left( \frac{|\boldsymbol{x}|}{\alpha} \right) = \boldsymbol{K}(\boldsymbol{x}) B_K \left( \frac{|\boldsymbol{x}|}{\alpha} \right),
\end{equation*}
where $B_K(r) = 1 - r K_1(r)$ and $K_1$ is the first order modified Bessel function of the second kind. Similar to the vortex blob method, we have the relation,
\begin{equation*}
B_K(r) = - 2 \pi r \frac{dG^{\alpha=1}_r}{dr}, \qquad B_K'(r) = r K_0(r). 
\end{equation*}
Especially, it follows from the properties of the Bessel function that $G^\alpha$ is bounded and $\boldsymbol{K}^\alpha(0) = 0$, namely they have no singularity.

\label{alpha}

\subsection{Main theorems}
Our main results consist of two theorems, the existence of a unique global weak solution for (\ref{RPDE}) in the space of Radon measure and its convergence to the Euler equations with the integrable and bounded initial vorticity. Recall that $\mathcal{M}(\mathbb{R}^2)$ denotes the space of finite Radon measures on $\mathbb{R}^2$ with the norm,
\begin{equation*}
\| \mu \|_{\mathcal{M}} = \sup \left\{ \left. \int_{\mathbb{R}^2} f d \mu \ \right| \ f  \in C_0(\mathbb{R}^2), \ \| f \|_{L^\infty} \leq 1 \right\},
\end{equation*}
where $C_0(\mathbb{R}^2)$ is the space of continuous functions vanishing at infinity. For later use, we define functions as follows. 
\begin{equation*}
\chi_{\log}^{-} (\boldsymbol{x})  = \left\{
\begin{array}{cc}
\displaystyle{ \left( 1 - \log{|\boldsymbol{x}|} \right)^{-1} }  &, \ |\boldsymbol{x}| \leq 1 , \\
\displaystyle{ 0 }  &, \ |\boldsymbol{x}| > 1,
\end{array}
\right. \quad 
\chi_{\log}^{+}(\boldsymbol{x})  = \left\{
\begin{array}{cc}
\displaystyle{0}  &,  \  |\boldsymbol{x}| \leq 1, \\
\displaystyle{1 + \log{|\boldsymbol{x}|} } &, \ |\boldsymbol{x}| > 1 , 
\end{array}
\right. 
\end{equation*}
and
\begin{equation*}
\chi^{-}_{\alpha}(\boldsymbol{x})  = \left\{
\begin{array}{cc}
|\boldsymbol{x}|^\alpha  &,  \ |\boldsymbol{x}| \leq 1 , \\
0  &,  \ |\boldsymbol{x}| > 1,
\end{array}
\right. \qquad 
\chi^{+}_{\alpha}(\boldsymbol{x}) = \left\{
\begin{array}{cc}
0   &,  \ |\boldsymbol{x}| \leq 1, \\
|\boldsymbol{x}|^\alpha  &,  \ |\boldsymbol{x}| > 1 .
\end{array}
\right. 
\end{equation*}
We also set $\chi_{\alpha}(\boldsymbol{x}) = \chi_{\alpha}^{-}(\boldsymbol{x}) + \chi_{\alpha}^{+}(\boldsymbol{x}) =  |\boldsymbol{x}|^\alpha$ for $\boldsymbol{x}\in \mathbb{R}^2$. These functions are used to characterize singularities and decay rates of functions. Then, we have the following theorem.

\begin{thm}
Suppose that $h \in C^1(\mathbb{R}^2) \cap W^1_1(\mathbb{R}^2)$ satisfies $\chi_1^{+} h \in L^1(\mathbb{R}^2)$ and
\begin{equation}
\chi_{\log}^{-} h \in L^\infty(\mathbb{R}^2), \qquad  \chi_1 \nabla h \in L^\infty(\mathbb{R}^2). \label{h-asympt}
\end{equation}
Then, for any initial vorticity $q_0 \in \mathcal{M}(\mathbb{R}^2)$, there exists a unique global weak solution of (\ref{RPDE}) such that 
\begin{equation*}
\boldsymbol{\eta}_h \in C^1(\mathbb{R};\mathscr{G}), \quad \boldsymbol{u}_h \in C(\mathbb{R}; C (\mathbb{R}^2; \mathbb{R}^2)), \quad q \in C(\mathbb{R};\mathcal{M}(\mathbb{R}^2)),
\end{equation*}
where $\mathscr{G}$ denotes the group of all homeomorphism of $\mathbb{R}^2$ that preserves the Lebesgue measure. \label{well-posed}
\end{thm}

\begin{rem}
The assumption (\ref{h-asympt}) allows the smoothing function $h$ and $\nabla h$ to have singularities and behave 
\begin{equation*}
h(\boldsymbol{x}) \sim  \mathcal{O}\left( - \log |\boldsymbol{x}| \right), \qquad \nabla h(\boldsymbol{x}) \sim  \mathcal{O}\left( |\boldsymbol{x}|^{-1} \right), 
\end{equation*}
as $|\boldsymbol{x}| \rightarrow 0$. The decay rates of them are estimated by the conditions $\chi_1^{+} h$, $\nabla h \in L^1$.
\end{rem}

In order to consider the convergence of solutions of (\ref{RPDE}) to those in the Euler equations, we consider the parameterized function $h^\varepsilon$ defined by 
\begin{equation*}
h^\varepsilon(\boldsymbol{x}) = \frac{1}{2 \pi \varepsilon^2} h\left( \frac{\boldsymbol{y}}{\varepsilon} \right).
\end{equation*}
Moreover, we assume that $h$ is a radial function, namely it depends only on $r = |\boldsymbol{x}|$, and define $h_r(|\boldsymbol{x}|) = h(\boldsymbol{x})$ that is a function on $\mathbb{R}$. Then, we have the following parameterized solutions $(\boldsymbol{\eta}^\varepsilon, \boldsymbol{u}^\varepsilon, q)$ such that 
\begin{align*}
\partial_t \boldsymbol{\eta}^\varepsilon(\boldsymbol{x}, t) &= \boldsymbol{u}^\varepsilon \left( \boldsymbol{\eta}^\varepsilon(\boldsymbol{x}, t), t \right), \quad \boldsymbol{\eta}^\varepsilon (\boldsymbol{x}, 0) = \boldsymbol{x}, \\
&q(\boldsymbol{x}, t) = q_0 \left( \boldsymbol{\eta}^\varepsilon(\boldsymbol{x}, -t) \right), \\
\boldsymbol{u}^\varepsilon &=  \boldsymbol{K}^\varepsilon \ast q, \quad \boldsymbol{K}^\varepsilon = \boldsymbol{K} \ast h^\varepsilon,
\end{align*}
and the following theorem about the convergence to the Euler equations in the $\varepsilon \rightarrow 0$ limit holds.

\begin{thm}
In addition to the assumptions in Theorem~\ref{well-posed}, suppose that $h$ is radial and satisfies
\begin{equation*}
\int_0^\infty k h_r(k) dk = 1.
\end{equation*}
Let $q_0 = \omega_0 \in L^1(\mathbb{R}^2) \cap L^\infty(\mathbb{R}^2)$. Then, for any $T > 0$, there exists $C(T) > 0$ such that 
\begin{equation*}
\sup_{t\in[0,T]} \sup_{\boldsymbol{x}\in\mathbb{R}^2} \left| \boldsymbol{\eta}^\varepsilon(\boldsymbol{x}, t) - \boldsymbol{\eta}(\boldsymbol{x}, t) \right| \leq C(T) \varepsilon^{e^{-T}}.
\end{equation*}
\label{convergence}
\end{thm}

\begin{rem}
We remark that both the smoothing functions $h^\delta$ and $h^\alpha$ satisfy the assumptions of Theorem~\ref{well-posed} and Theorem~\ref{convergence}. Especially, the Euler-$\alpha$ model is a critical regularization in terms of the singularity since we know $K_0(r) \sim - \log{r}$ as $r \rightarrow 0$ in the smoothing function $h^\alpha$. On the other hand, it is easily confirmed that $h^\delta$ has a slowest algebraic decay to satisfy the assumptions of the theorems.

\end{rem}

\label{Main}

\section{Properties of the integral kernel}
In order to show the main theorems, it is important to investigate the properties of the integral kernel $\boldsymbol{K}_h$ and $\boldsymbol{K}^\varepsilon$ in the regularized Bio-Savart law. Let us first define the function $\varphi(r)$ for $r \geq 0$ by
\begin{equation*}
\varphi(r) = \left\{
\begin{array}{cc}
r(1 - \log{r}) \ &, \quad r < 1 , \\
1 \ &, \quad r \geq 1.
\end{array}
\right.
\end{equation*}
We now show that $\boldsymbol{K}_h$ is well-defined and especially $\boldsymbol{K}_h$ vanishes at the origin. Here, the smoothing function $h$ may have a singularity only at the origin. We actually assume that $h \in L^\infty(\mathbb{R}^2 \backslash B_\delta)$ with sufficiently small $\delta > 0$, where $B_R = \{ \boldsymbol{y}\in \mathbb{R}^2 \ | \ |\boldsymbol{y}| < R \}$.

\begin{lem}
Suppose that $h \in L^1(\mathbb{R}^2) \cap L^\infty(\mathbb{R}^2 \backslash B_\delta)$, $\delta > 0$ and $\chi_\alpha^{-} h \in L^\infty(\mathbb{R}^2)$, $\alpha\in [0,1)$. Then, $\boldsymbol{K}_h$ is uniformly bounded, namely $\boldsymbol{K}_h \in L^\infty(\mathbb{R}^2)$. Moreover, if $h$ is radial, then $\boldsymbol{K}_h(0) = 0$.  \label{K_h-bdd}
\end{lem}

\begin{proof}
We have  
\begin{align*}
\left| \boldsymbol{K}_h(\boldsymbol{x}) \right| &\leq \frac{1}{2 \pi} \int_{\mathbb{R}^2} \frac{1}{|\boldsymbol{x} - \boldsymbol{y}|} |h(\boldsymbol{y})| d\boldsymbol{y} \\
& \leq \frac{1}{2 \pi} \int_{|\boldsymbol{x} - \boldsymbol{y}| < \delta } \frac{1}{|\boldsymbol{x} - \boldsymbol{y}|} |h(\boldsymbol{y})| d \boldsymbol{y} + \frac{1}{2 \pi} \int_{|\boldsymbol{x} - \boldsymbol{y}| \geq \delta } \frac{1}{|\boldsymbol{x} - \boldsymbol{y}|} |h(\boldsymbol{y})| d \boldsymbol{y}   \\
& \leq \frac{1}{2 \pi} \| \chi_\alpha^{-} h \|_{L^\infty} \left[ \int_{|\boldsymbol{y}| < |\boldsymbol{x} - \boldsymbol{y}| < \delta }  \frac{1}{|\boldsymbol{y}|^{1+\alpha}} d  \boldsymbol{y} + \int_{ |\boldsymbol{x} - \boldsymbol{y}| \leq |\boldsymbol{y}| < \delta } \frac{1}{|\boldsymbol{x} - \boldsymbol{y}|^{1+\alpha}} d \boldsymbol{y} \right] \\
&\qquad  + \frac{1}{2 \pi} \| h \|_{L^\infty(\mathbb{R}^2 \backslash B_\delta)} \int_{|\boldsymbol{x} - \boldsymbol{y}| < \delta \cap |\boldsymbol{y}| \geq \delta } \frac{1}{|\boldsymbol{x} - \boldsymbol{y}|} d \boldsymbol{y} +  \frac{1}{2 \pi \delta} \int_{\mathbb{R}^2 } |h(\boldsymbol{y})| d \boldsymbol{y}  \\
&\leq \frac{2 \delta^{1-\alpha}}{1 - \alpha} \| \chi_\alpha^{-} h \|_{L^\infty} + \delta \| h \|_{L^\infty(\mathbb{R}^2 \backslash B_\delta)} + \frac{1}{2 \pi \delta} \| h \|_{L^1}, 
\end{align*}
in which $\| \cdot \|_{L^p} = \| \cdot \|_{L^p(\mathbb{R}^2)}$ and this notation is used throughout this paper. Thus, $\boldsymbol{K}_h$ is uniformly bounded. In terms of the polar coordinates, the value of $\boldsymbol{K}_h$ at the origin is calculated as follows.
\begin{align*}
\boldsymbol{K}_h(0) &= - \frac{1}{2 \pi} \int_{\mathbb{R}^2} \frac{(y_2, - y_1)}{|\boldsymbol{y}|^2} h(\boldsymbol{y}) d\boldsymbol{y} \\
&= - \frac{1}{2 \pi} \int_0^\infty \int_0^{2 \pi}  \frac{(k \sin{\theta}, - k \cos{\theta}) }{k^2} h_r(k) k d\theta dk \\
&= - \frac{1}{2 \pi} \int_0^\infty h_r(k) dk \int_0^{2 \pi}  (\sin{\theta}, - \cos{\theta}) d\theta,  
\end{align*}
Since $h_r$ is integrable, we conclude that $\boldsymbol{K}_h(0) = 0$.
\end{proof}

\begin{rem}
When we show that $\boldsymbol{K}_h$ is bounded outside of a neighborhood of the origin, it is enough to suppose that $h \in L^1(\mathbb{R}^2) \cap L^\infty(\mathbb{R}^2 \backslash B_\delta)$.
\end{rem}

Next, we see the asymptotic behavior of $\boldsymbol{K}_h$. The following lemma asserts that it decays at least as fast as the singular kernel $\boldsymbol{K}$.

\begin{lem}
For $h \in L^1(\mathbb{R}^2)$ satisfying $\chi_1^{+} h \in L^1(\mathbb{R}^2) \cap L^\infty(\mathbb{R}^2)$, we have
\begin{equation*}
\limsup_{|\boldsymbol{x}| \rightarrow \infty} \left| \boldsymbol{x} \right| \left|\boldsymbol{K}_h(\boldsymbol{x}) \right| \leq \frac{1}{2 \pi} \| h \|_{L^1}.
\end{equation*}
\label{K_h-decay}
\end{lem}

\begin{proof}
We divide the integral in $\boldsymbol{K}_h$ as follows.
\begin{align*}
\left| \boldsymbol{x} \right| \left| \boldsymbol{K}_h (\boldsymbol{x}) \right| &\leq \frac{1}{2 \pi}\int_{\mathbb{R}^2} \left( 1 + \frac{| \boldsymbol{y} |}{|\boldsymbol{x} - \boldsymbol{y}|} \right) \left| h(\boldsymbol{y}) \right| d\boldsymbol{y} \\
& = \frac{1}{2\pi} \left[ \| h \|_{L^1} + \int_{B_R} \frac{| \boldsymbol{y} |}{|\boldsymbol{x} - \boldsymbol{y}|} \left| h(\boldsymbol{y}) \right| d\boldsymbol{y} + \int_{\mathbb{R}^2 \backslash B_R} \frac{| \boldsymbol{y} |}{|\boldsymbol{x} - \boldsymbol{y}|} \left| h(\boldsymbol{y}) \right| d\boldsymbol{y} \right]. 
\end{align*}
The second term in the right-hand side vanishes in the $|\boldsymbol{x}| \rightarrow \infty$ limit, since we have 
\begin{equation*}
\int_{B_R} \frac{| \boldsymbol{y} |}{|\boldsymbol{x} - \boldsymbol{y}|} \left| h(\boldsymbol{y}) \right| d\boldsymbol{y} \leq \frac{R}{|\boldsymbol{x}| - R} \int_{B_R} \left| h(\boldsymbol{y}) \right| d\boldsymbol{y},
\end{equation*}
for sufficiently large $|\boldsymbol{x}|$. The third term is estimated as follows.
\begin{align*}
\int_{\mathbb{R}^2 \backslash B_R} \frac{| \boldsymbol{y} |}{|\boldsymbol{x} - \boldsymbol{y}|} \left| h(\boldsymbol{y}) \right| d\boldsymbol{y} &= \int_{(\mathbb{R}^2 \backslash B_R) \cap |\boldsymbol{x} - \boldsymbol{y}| < \delta} \frac{| \boldsymbol{y} |}{|\boldsymbol{x} - \boldsymbol{y}|} \left| h(\boldsymbol{y}) \right| d\boldsymbol{y} \\
& \quad  +  \int_{(\mathbb{R}^2 \backslash B_R) \cap |\boldsymbol{x} - \boldsymbol{y}| \geq \delta} \frac{| \boldsymbol{y} |}{|\boldsymbol{x} - \boldsymbol{y}|} \left| h(\boldsymbol{y}) \right| d\boldsymbol{y} \\
&\leq  2 \pi \delta \| \chi_1^{+} h \|_{L^\infty(\mathbb{R}^2 \backslash B_R)} + \frac{1}{\delta} \| \chi_1^{+} h \|_{L^1(\mathbb{R}^2 \backslash B_R)}.
\end{align*}
Summarizing the above estimates, we find
\begin{equation*}
\limsup_{|\boldsymbol{x}| \rightarrow \infty} \left| \boldsymbol{x} \right| \left| \boldsymbol{K}_h (\boldsymbol{x}) \right| \leq  \frac{1}{2 \pi} \| h \|_{L^1} + \delta \| \chi_1^{+} h \|_{L^\infty(\mathbb{R}^2 \backslash B_1)} + \frac{1}{2 \pi \delta} \| \chi_1^{+} h \|_{L^1(\mathbb{R}^2 \backslash B_R)}.
\end{equation*}
Since $\chi_1^{+} h$ is integrable, the third term vanishes in the $R \rightarrow \infty$ limit. Finally, taking the $\delta \rightarrow 0$ limit, we obtain the desired result.
\end{proof}

In the following lemma, we show that $\boldsymbol{K}_h$ is quasi-Lipschitz continuous.

\begin{lem}
Let $h \in C^1(\mathbb{R}^2) \cap W^1_1(\mathbb{R}^2)$ satisfies $\chi_{\log}^{-} h \in L^\infty(\mathbb{R}^2)$ and $\chi_1 \nabla h \in L^\infty(\mathbb{R}^2)$. Then, 
\begin{equation*}
\left| \boldsymbol{K}_h (\boldsymbol{x}) - \boldsymbol{K}_h(\boldsymbol{x}') \right| \leq c \varphi\left( \left| \boldsymbol{x} - \boldsymbol{x}' \right| \right).
\end{equation*}
\label{Kh-est}
\end{lem}

\begin{proof}
Let us set $r = |\boldsymbol{x} - \boldsymbol{x}'|$ and note that it is enough to prove the lemma for sufficiently small $r > 0$. We have
\begin{align*}
\left| \boldsymbol{K}_h (\boldsymbol{x}) - \boldsymbol{K}_h(\boldsymbol{x}') \right| & = \left| \int_{\mathbb{R}^2} \boldsymbol{K}(\boldsymbol{y}) \left(  h(\boldsymbol{x} - \boldsymbol{y}) - h(\boldsymbol{x}' - \boldsymbol{y}) \right) d \boldsymbol{y} \right| \\
& \leq \int_{|\boldsymbol{x} - \boldsymbol{y}| < 2 r} \left| \boldsymbol{K}(\boldsymbol{y}) \right| \left(  \left| h(\boldsymbol{x} - \boldsymbol{y}) \right| + \left| h(\boldsymbol{x}' - \boldsymbol{y}) \right| \right) d\boldsymbol{y} \\
&\quad + \int_{2r \leq |\boldsymbol{x} - \boldsymbol{y}| < 2} \left| \boldsymbol{K}(\boldsymbol{y}) \right| \left| h(\boldsymbol{x} - \boldsymbol{y}) - h(\boldsymbol{x}' - \boldsymbol{y}) \right| d\boldsymbol{y}  \\
&\quad + \int_{|\boldsymbol{x} - \boldsymbol{y}| \geq 2} \left| \boldsymbol{K}(\boldsymbol{y}) \right| \left| h(\boldsymbol{x} - \boldsymbol{y}) - h(\boldsymbol{x}' - \boldsymbol{y}) \right| d\boldsymbol{y}  \\
&\equiv I_1 + I_2 + I_3.
\end{align*}
Considering the fact that if $|\boldsymbol{x} - \boldsymbol{y}| < 2 r$, then $|\boldsymbol{x}' - \boldsymbol{y}| < 3 r$, we obtain 
\begin{align*}
I_1 &\leq \frac{1}{\pi} \sup_{\boldsymbol{x} \in \mathbb{R}^2} \int_{|\boldsymbol{x} - \boldsymbol{y}| < 3 r} \frac{1}{|\boldsymbol{y}|} \left| h\left( \boldsymbol{x} - \boldsymbol{y} \right) \right| d\boldsymbol{y}  \\
&\leq \frac{1}{\pi} \| \chi_{\log}^{-} h \|_{L^\infty} \sup_{\boldsymbol{x} \in \mathbb{R}^2} \int_{|\boldsymbol{x} - \boldsymbol{y}| < 3 r} \frac{ 1 - \log{|\boldsymbol{x} - \boldsymbol{y}|}}{|\boldsymbol{y}|} d\boldsymbol{y} \\
&\leq \frac{1}{\pi} \| \chi_{\log}^{-} h \|_{L^\infty} \sup_{\boldsymbol{x} \in \mathbb{R}^2} \left[ \int_{|\boldsymbol{y}| < |\boldsymbol{x} - \boldsymbol{y}| < 3 r} \frac{ 1 - \log{|\boldsymbol{y}|}}{|\boldsymbol{y}|} d\boldsymbol{y} \right.\\
& \left. \hspace{40mm} + \int_{ |\boldsymbol{x} - \boldsymbol{y}| \leq |\boldsymbol{y}|\, \cap\, |\boldsymbol{x} - \boldsymbol{y}| < 3 r } \frac{ 1 - \log{|\boldsymbol{x} - \boldsymbol{y}|}}{|\boldsymbol{x} - \boldsymbol{y}|} d\boldsymbol{y} \right] \\
&\leq c_1 \varphi(r) \| \chi_{\log}^{-} h \|_{L^\infty}.
\end{align*}
In order to estimate $I_2$ and $I_3$, we apply the mean value theorem to $h$. Then, we obtain 
\begin{equation*}
\left| h(\boldsymbol{x} - \boldsymbol{y}) - h(\boldsymbol{x}' - \boldsymbol{y}) \right| \leq r \int_0^1 \left| (\nabla h)( \boldsymbol{x} - \boldsymbol{y} + \tau (\boldsymbol{x}' - \boldsymbol{x}) ) \right| d\tau.
\end{equation*}
Moreover, it follows from $|\boldsymbol{x} - \boldsymbol{y}| \geq 2r$ that
\begin{equation}
|\boldsymbol{x} - \boldsymbol{y} + \tau (\boldsymbol{x}' - \boldsymbol{x})| \geq \left| |\boldsymbol{x} - \boldsymbol{y}| - \tau r \right| \geq \frac{1}{2}|\boldsymbol{x} - \boldsymbol{y}| + (1 - \tau) r \geq \frac{1}{2}|\boldsymbol{x} - \boldsymbol{y}|. \label{x-y}
\end{equation}
Thus, setting $D = \left\{ \boldsymbol{y}\in\mathbb{R}^2 \ |\  2 r \leq |\boldsymbol{x} - \boldsymbol{y}| < 2 \right\}$, we find
\begin{align*}
I_2 &\leq \frac{r}{2 \pi} \| \chi_1 \nabla h \|_{L^\infty} \int_D \frac{1}{|\boldsymbol{y}|} \int_0^1 \frac{1}{|\boldsymbol{x} - \boldsymbol{y} + \tau (\boldsymbol{x}' - \boldsymbol{x})|} d\tau d\boldsymbol{y} \\
&\leq \frac{r}{\pi} \| \chi_1 \nabla h \|_{L^\infty} \int_D \frac{1}{|\boldsymbol{y}||\boldsymbol{x} - \boldsymbol{y}|} d\boldsymbol{y} \\
&\leq \frac{r}{\pi} \| \chi_1 \nabla h \|_{L^\infty} \left[ \frac{1}{2 r} \int_{D\, \cap\, |\boldsymbol{y}| < 2 r} \frac{1}{|\boldsymbol{y}|} d\boldsymbol{y} + \int_{D\, \cap\, 2 r \leq |\boldsymbol{y}| < |\boldsymbol{x} - \boldsymbol{y}|} \frac{1}{|\boldsymbol{y}|^2} d\boldsymbol{y} \right. \\ 
&\hspace{65mm}  + \left. \int_{D\, \cap\, |\boldsymbol{x} - \boldsymbol{y}| \leq |\boldsymbol{y}|} \frac{1}{|\boldsymbol{x} - \boldsymbol{y}|^2} d\boldsymbol{y} \right] \\
&\leq 2 r \| \chi_1 \nabla h \|_{L^\infty} \left[ \frac{1}{2 r} \int_0^{2r} dk + 2 \int_{2r}^{2} \frac{1}{k} dk \right] \\
&\leq c_2 \varphi(r) \| \chi_1 \nabla h \|_{L^\infty}.
\end{align*}
Finally, it follows from (\ref{x-y}) that
\begin{align*}
I_3 &\leq \frac{r}{2 \pi} \int_{2 \leq |\boldsymbol{x} - \boldsymbol{y}|} \frac{1}{|\boldsymbol{y}|} \int_0^1 \left| (\nabla h) (\boldsymbol{x} - \boldsymbol{y} + \tau (\boldsymbol{x}' - \boldsymbol{x})) \right| d\tau d\boldsymbol{y} \\
&\leq \frac{r}{2 \pi} \left[  \| \chi_1 \nabla h \|_{L^\infty} \int_{2 \leq |\boldsymbol{x} - \boldsymbol{y}|\, \cap\, |\boldsymbol{y}| < 1} \frac{2}{|\boldsymbol{y}| |\boldsymbol{x} - \boldsymbol{y}|} d\boldsymbol{y} \right. \\
&\hspace{15mm} + \left. \int_{2 \leq |\boldsymbol{x} - \boldsymbol{y}|\, \cap\, |\boldsymbol{y}| \geq 1 } \int_0^1 \left| (\nabla h) (\boldsymbol{x} - \boldsymbol{y} + \tau (\boldsymbol{x}' - \boldsymbol{x})) \right| d\tau d\boldsymbol{y} \right] \\
&\leq r \left[ \| \chi_1 \nabla h \|_{L^\infty} + \frac{1}{2 \pi} \| \nabla h \|_{L^1} \right]. 
\end{align*}
Combining the above three estimates, we achieve the conclusion.
\end{proof}

\begin{rem}
As a consequence of above lemmas, we find $\boldsymbol{K}_h \in C_0(\mathbb{R}^2)$. This plays an essential role to show the existence of a unique global solution of (\ref{RPDE}) with initial vorticity in $\mathcal{M}(\mathbb{R}^2)$. Indeed, owing to the following estimate,
\begin{equation} 
\| \boldsymbol{u}_h \|_{L^\infty} \leq \sup_{\boldsymbol{x} \in \mathbb{R}^2} \left| \int_{\mathbb{R}^2} \boldsymbol{K}_h(\boldsymbol{x} - \boldsymbol{y}) q(\boldsymbol{y}) d\boldsymbol{y} \right| \leq \| \boldsymbol{K}_h \|_{L^\infty} \| q \|_{\mathcal{M}}, \label{unif-bdd}
\end{equation}
and Lemma~\ref{Kh-est}, we find that $\boldsymbol{u}_h$ is uniformly bounded and quasi-Lipschitz continuous , which allows us to apply the classical method to prove Theorem~\ref{well-posed}. \label{C_0}
\end{rem}

Finally, we investigate the relation between $\boldsymbol{K}^\varepsilon$ and $\boldsymbol{K}$ to prove Theorem~\ref{convergence}. Before that, we show the following lemma.
\begin{lem}
Suppose that $h \in L^1(\mathbb{R}^2) \cap L^\infty(\mathbb{R}^2 \backslash B_\delta)$, $\delta > 0$ satisfies $\chi_\alpha^{-} h \in L^\infty(\mathbb{R}^2)$, $\alpha\in[0,1]$, and $\chi_{\log}^{+} h \in L^1(\mathbb{R}^2)$. Then, the scalar function 
\begin{equation*}
G_h(\boldsymbol{x}) = - \frac{1}{2 \pi} \int_{\mathbb{R}^2} \left( \log{|\boldsymbol{x} - \boldsymbol{y}|} \right) h\left( \boldsymbol{y} \right) d\boldsymbol{y},
\end{equation*}
is well-defined for any $\boldsymbol{x}\in\mathbb{R}^2$ and satisfies the Poisson equation, $- \Delta G_h = h$, in the weak sense that $\langle \nabla G_h, \nabla\phi \rangle = - \langle h, \phi \rangle$ for any $\phi \in C_0^\infty(\mathbb{R}^2)$. Moreover, we have
\begin{equation}
\lim_{|\boldsymbol{x}| \rightarrow \infty} \frac{G_h(\boldsymbol{x})}{\log{|\boldsymbol{x}|}} =  - \frac{1}{2 \pi} \int_{\mathbb{R}^2} h(\boldsymbol{y}) d\boldsymbol{y}. \label{E_h-lim}
\end{equation}
Especially, if $h$ is a radial function, $G_h$ is also radial, namely $G_h(\boldsymbol{x}) = G_{h,r}(|\boldsymbol{x}|)$. 

Regarding the derivative of $G_{h,r}$, suppose that the radial function $h \in L^1(\mathbb{R}^2)$ satisfies $\chi_\alpha^{-} h \in L^\infty(\mathbb{R}^2)$, $\alpha\in[0,1)$ and $\chi_1^{+} h \in L^1(\mathbb{R}^2) \cap L^\infty(\mathbb{R}^2)$. Then, we have $G'_{h,r}(0) = 0$ and
\begin{equation}
\lim_{r \rightarrow \infty} r G'_{h,r}(r) = - \frac{1}{2 \pi} \int_{\mathbb{R}^2} h(\boldsymbol{y}) d\boldsymbol{y}. \label{G'_h-lim}
\end{equation}
\label{K_eps-lem}
\end{lem}

\begin{proof}
We first show that $G_h$ is well-defined. For any fixed $\boldsymbol{x}\in\mathbb{R}^2$, we divide the integral into two parts,
\begin{equation*}
\left| G_h(\boldsymbol{x}) \right| \leq \frac{1}{2 \pi} \left[ \int_{|\boldsymbol{x} - \boldsymbol{y}| < 1} \left| \log{|\boldsymbol{x} - \boldsymbol{y}|} \right| | h\left( \boldsymbol{y} \right) | d\boldsymbol{y} + \int_{|\boldsymbol{x} - \boldsymbol{y}| \geq 1} \left| \log{|\boldsymbol{x} - \boldsymbol{y}|} \right| | h\left( \boldsymbol{y} \right) | d\boldsymbol{y} \right].
\end{equation*}
Each integral is estimated as follows.
\begin{align*}
\int_{|\boldsymbol{x} - \boldsymbol{y}| < 1} \left| \log{|\boldsymbol{x} - \boldsymbol{y}|} \right| | h\left( \boldsymbol{y} \right) | d\boldsymbol{y} & \leq \| \chi_{\alpha}^{-} h \|_{L^\infty} \left[ \int_{ |\boldsymbol{y}| < |\boldsymbol{x} - \boldsymbol{y}| < 1 \, \cap\, |\boldsymbol{y}| < \delta} \frac{ - \log{|\boldsymbol{y}|} }{|\boldsymbol{y}|^\alpha} d\boldsymbol{y} \right. \\
& \hspace{25mm} \left. + \int_{ |\boldsymbol{x} - \boldsymbol{y}| \leq |\boldsymbol{y}| < \delta} \frac{ - \log{|\boldsymbol{x} - \boldsymbol{y}|} }{|\boldsymbol{x} - \boldsymbol{y}|^\alpha} d\boldsymbol{y}\right] \\
& \quad + \| h \|_{L^\infty(\mathbb{R}^2 \backslash B_\delta)} \int_{|\boldsymbol{x} - \boldsymbol{y}| < 1 \, \cap\, |\boldsymbol{y}| \geq \delta} \left( - \log{|\boldsymbol{x} - \boldsymbol{y}|} \right) d\boldsymbol{y} \\
&\leq 4 \pi \| \chi_{\alpha}^{-} h \|_{L^\infty} \int_0^\delta k^{1-\alpha} ( - \log{k}) dk \\
& \quad + 2 \pi \| h \|_{L^\infty(\mathbb{R}^2 \backslash B_\delta)} \int_0^1 k ( - \log{k}) dk\\
&\leq c \left(  \| \chi_{\alpha}^{-} h \|_{L^\infty} + \| h \|_{L^\infty(\mathbb{R}^2 \backslash B_\delta)} \right),
\end{align*}
and 
\begin{align*}
\int_{|\boldsymbol{x} - \boldsymbol{y}| \geq 1} \left| \log{|\boldsymbol{x} - \boldsymbol{y}|} \right| | h\left( \boldsymbol{y} \right) | d\boldsymbol{y} & \leq \int_{|\boldsymbol{x} - \boldsymbol{y}| \geq 1} \log{( |\boldsymbol{x}| + |\boldsymbol{y}|)} | h\left( \boldsymbol{y} \right) | d\boldsymbol{y} \\
&\leq \chi_{\{|\boldsymbol{x}| \geq 1/2 \}}\int_{|\boldsymbol{y}| < |\boldsymbol{x}|} \log{(2|\boldsymbol{x}|)} | h\left( \boldsymbol{y} \right) | d\boldsymbol{y} \\
&\qquad  + \int_{|\boldsymbol{y}| \geq |\boldsymbol{x}|\, \cap\, |\boldsymbol{y}| \geq 1/2} \log{(2|\boldsymbol{y}|)} | h\left( \boldsymbol{y} \right) | d\boldsymbol{y} \\
&\leq \chi_{\log}^{+}(\boldsymbol{x}) \| h \|_{L^1} + c\left(  \| h \|_{L^1} +  \| \chi_{\log}^{+} h \|_{L^1} \right). 
\end{align*}
Thus, we find
\begin{equation}
\left| G_h(\boldsymbol{x}) \right| \leq  \chi_{\log}^{+}(\boldsymbol{x}) \| h \|_{L^1} + c \left(  \| h \|_{L^1} +  \| \chi_{\log}^{+} h \|_{L^1} + \| \chi_{\alpha}^{-} h \|_{L^\infty} + \| h \|_{L^\infty(\mathbb{R}^2 \backslash B_\delta)} \right). \label{E_h-WD}
\end{equation}
In order to see an asymptotic behavior of $G_h$, we decompose the integral as follows. 
\begin{align*}
\frac{G_h(\boldsymbol{x})}{\log{|\boldsymbol{x}|}} = - \frac{1}{2 \pi} \left[ \int_{B_R} \frac{\log{|\boldsymbol{x} - \boldsymbol{y}|}}{\log{|\boldsymbol{x}|}} h(\boldsymbol{y}) d\boldsymbol{y} + \int_{\mathbb{R}^2 \backslash B_R} \frac{\log{|\boldsymbol{x} - \boldsymbol{y}|}}{\log{|\boldsymbol{x}|}} h(\boldsymbol{y}) d\boldsymbol{y} \right]
\end{align*}
As for the first term, since we have the pointwise convergence, $\log{|\boldsymbol{x} - \boldsymbol{y}|} / \log{|\boldsymbol{x}|} \rightarrow 1$ as $| \boldsymbol{x} | \rightarrow \infty$, and 
\begin{equation*}
\frac{\log{|\boldsymbol{x} - \boldsymbol{y}|}}{\log{|\boldsymbol{x}|}} |h(\boldsymbol{y})| \leq c |h(\boldsymbol{y})|,
\end{equation*}
for $\boldsymbol{y} \in B_{R}$ and sufficiently large $|\boldsymbol{x}|$, the dominated convergence theorem gives
\begin{equation*}  
 \lim_{|\boldsymbol{x}| \rightarrow \infty} \int_{B_R} \frac{\log{|\boldsymbol{x} - \boldsymbol{y}|}}{\log{|\boldsymbol{x}|}} h(\boldsymbol{y}) d\boldsymbol{y} = \int_{B_R} h(\boldsymbol{y}) d\boldsymbol{y}.
\end{equation*}
Based on the calculations in (\ref{E_h-WD}), we find the following estimate for sufficiently large $|\boldsymbol{x}|$. 
\begin{align*}
\left| \int_{\mathbb{R}^2 \backslash B_R} \frac{\log{|\boldsymbol{x} - \boldsymbol{y}|}}{\log{|\boldsymbol{x}|}} h(\boldsymbol{y}) d\boldsymbol{y} \right| & \leq \| h \|_{L^1(\mathbb{R}^2 \backslash B_R)} \\
& \quad + \frac{c}{\log{|\boldsymbol{x}|}} \left(  \| h \|_{L^1} +  \| \chi_{\log}^{+} h \|_{L^1} + \| h \|_{L^\infty(\mathbb{R}^2 \backslash B_\delta)} \right).
\end{align*}
Thus, taking the $|\boldsymbol{x}| \rightarrow \infty$ limit and then the $R \rightarrow \infty$ limit, we obtain the desired result (\ref{E_h-lim}). According to the classical theory, we readily confirm that $G_h$ is a weak solution of the Poisson equation, see \cite{Masaki}. Next, we rewrite $G_h$ in the polar coordinates,
\begin{equation*}
G_{h,r}(r, \theta) = - \frac{1}{2 \pi} \int_{\mathbb{R}^2} \left( \log{\sqrt{(r \cos{\theta} - y_1)^2 + (r \sin{\theta} - y_2)^2}} \right) h(\boldsymbol{y}) d\boldsymbol{y}.
\end{equation*}
Then, for a radial function $h$, we have
\begin{align*}
G_{h,r}(r, \theta) &= - \frac{1}{2 \pi} \int_0^\infty \int_0^{2 \pi} \left( \log{ \sqrt{ (r \cos{\theta} - \tilde{r}\cos{\tilde{\theta}})^2 + (r \sin{\theta} - \tilde{r}\sin{\tilde{\theta}})^2 } } \right) \tilde{r} h_r(\tilde{r}) d\tilde{\theta} d\tilde{r} \\
&= - \frac{1}{2 \pi} \int_0^\infty \int_0^{2 \pi} \left( \log{\sqrt{ r^2 + \tilde{r}^2 - 2 r \tilde{r}\cos{(\theta - \tilde{\theta})} }} \right) \tilde{r} h_r(\tilde{r}) d\tilde{\theta} d\tilde{r} \\
&= - \frac{1}{2 \pi} \int_0^\infty \int_0^{2 \pi} \left( \log{\sqrt{ r^2 + \tilde{r}^2 - 2 r \tilde{r}\cos{\tilde{\theta}} }} \right) \tilde{r} h_r(\tilde{r}) d\tilde{\theta} d\tilde{r}.
\end{align*}
Thus, $G_{h,r}$ is independent of $\theta$ and represented by
\begin{equation*}
G_{h,r}(r) = G_{h,r}(r, 0) = - \frac{1}{2 \pi} \int_{\mathbb{R}^2} \left( \log{\sqrt{(r - y_1)^2 + y_2^2}} \right) h(\boldsymbol{y}) d\boldsymbol{y}.
\end{equation*}
Its derivative is given by 
\begin{equation*}
G'_{h,r}(r) = - \frac{1}{2 \pi} \int_{\mathbb{R}^2} \frac{r - y_1}{(r - y_1)^2 + y_2^2} h(\boldsymbol{y}) d\boldsymbol{y},
\end{equation*}
and we easily find that $G'_{h,r}$ is uniformly bounded and $G'_{h,r}(0) = 0$ based on the calculation in Lemma~\ref{K_h-bdd}. Finally, we see the asymptotic behavior of $G'_{h,r}$. Note that
\begin{equation*}
r G'_{h,r}(r) = - \frac{1}{2 \pi} \int_{\mathbb{R}^2} h(\boldsymbol{y}) d\boldsymbol{y} + \frac{1}{2 \pi} \int_{\mathbb{R}^2} \frac{ - y_1(r - y_1) + y_2^2}{(r - y_1)^2 + y_2^2} h(\boldsymbol{y}) d\boldsymbol{y}.
\end{equation*}
Then, setting $\boldsymbol{x}_r = (r, 0)$, we obtain 
\begin{equation*}
\left| \int_{\mathbb{R}^2} \frac{- y_1(r - y_1) + y_2^2}{(r - y_1)^2 + y_2^2} h(\boldsymbol{y}) d\boldsymbol{y} \right| \leq \int_{\mathbb{R}^2} \frac{|\boldsymbol{y}|}{| \boldsymbol{x}_r - \boldsymbol{y} |} |h(\boldsymbol{y})| d\boldsymbol{y} \quad \rightarrow \quad 0,
\end{equation*}
as $r \rightarrow \infty$ in the same way as the proof in Lemma~\ref{K_h-decay}. Hence, we obtain (\ref{G'_h-lim}) and complete the proof.
\end{proof}

\begin{rem}
Under the condition of Theorem~\ref{convergence}, it follows from Lemma~\ref{K_eps-lem} that $G^\varepsilon = G \ast h^\varepsilon$ is radial, namely $G^\varepsilon(\boldsymbol{x}) = G_r^\varepsilon(|\boldsymbol{x}|)$ and 
\begin{equation}
\lim_{r \rightarrow \infty} r \frac{d G_r^1}{dr}(r) = - \frac{1}{2 \pi}\int_0^\infty k h_r(k)dk = - \frac{1}{2 \pi}. \label{G-lim}
\end{equation}
Note that the relation,
\begin{equation*}
G_r^\varepsilon(|\boldsymbol{x}|) = G_r^1\left( \frac{|\boldsymbol{x}|}{\varepsilon} \right) - \frac{1}{2 \pi}\log{\varepsilon}. 
\end{equation*}
Then, we find
\begin{equation*}
\boldsymbol{K}^\varepsilon(\boldsymbol{x}) = \nabla^\perp G_r^\varepsilon(|\boldsymbol{x}|)  =  \boldsymbol{K}(\boldsymbol{x}) \mathcal{G}\left( \frac{|\boldsymbol{x}|}{\varepsilon} \right) , \qquad \mathcal{G}(r) =  - 2 \pi r \frac{dG_r^1}{dr} (r).
\end{equation*}
Moreover, owing to (\ref{G-lim}) and $\frac{dG_r^1}{dr}(0) = 0$, we have
\begin{equation*}
\mathcal{G}(0) = 0, \qquad \lim_{r \rightarrow \infty} \mathcal{G}(r) = 1.
\end{equation*}
Hence, if $\boldsymbol{x} \neq 0$, then $\boldsymbol{K}^\varepsilon(\boldsymbol{x})$ converges pointwise to $\boldsymbol{K}(\boldsymbol{x})$.
\end{rem}

\begin{prop}
Let $h \in L^1(\mathbb{R}^2)$ be a radial function that satisfies $\chi_\alpha^{-} h \in L^\infty(\mathbb{R}^2)$, $\alpha\in[0,1)$, $\chi_1^{+} h \in L^1(\mathbb{R}^2) \cap L^\infty(\mathbb{R}^2)$ and $\int_0^\infty k h_r(k) dk = 1$. Then, we have
\begin{equation*}
\int_{\mathbb{R}^2} \left| \boldsymbol{K}^\varepsilon(\boldsymbol{x}) - \boldsymbol{K}(\boldsymbol{x}) \right| d\boldsymbol{x} \leq  \frac{\varepsilon}{2 \pi} \| \chi_1 h \|_{L^1}
\end{equation*}
\label{K_con}
\end{prop}

\begin{proof}
We first note that $h \in L^1(\mathbb{R}^2)$ and $\chi_1^{+} h \in L^1(\mathbb{R}^2)$ yields $\chi_1 h \in L^1(\mathbb{R}^2)$. We have
\begin{align*}
\int_{\mathbb{R}^2} \left| \boldsymbol{K}^\varepsilon(\boldsymbol{x}) - \boldsymbol{K}(\boldsymbol{x}) \right| d\boldsymbol{x} &= \int_{\mathbb{R}^2} \left| \int_{\mathbb{R}^2} \boldsymbol{K}(\boldsymbol{x} - \boldsymbol{y}) \frac{1}{2\pi\varepsilon^2} h\left( \frac{ \boldsymbol{y} }{\varepsilon} \right) d\boldsymbol{y} - \boldsymbol{K}(\boldsymbol{x}) \right| d\boldsymbol{x} \\
&= \varepsilon \int_{\mathbb{R}^2} \left| \frac{1}{2\pi} \int_{\mathbb{R}^2} \boldsymbol{K}(\boldsymbol{x} - \boldsymbol{y}) h\left( \boldsymbol{y} \right) d\boldsymbol{y} - \boldsymbol{K}(\boldsymbol{x}) \right| d\boldsymbol{x} \\
&= \frac{\varepsilon}{2 \pi} \int_{\mathbb{R}^2} \left| \nabla^\perp \left( G_h(\boldsymbol{x}) + \log{\left| \boldsymbol{x} \right|}  \right) \right| d\boldsymbol{x}. 
\end{align*}
Owing to Lemma~\ref{K_eps-lem}, $G_h$ is a radial function, namely $G_h(\boldsymbol{x}) = G_{h,r}(|\boldsymbol{x}|)$. Then, we find
\begin{equation*}
\nabla^\perp \left( G_h(\boldsymbol{x}) + \log{\left| \boldsymbol{x} \right|}  \right)  =  \frac{\boldsymbol{x}^\perp}{|\boldsymbol{x}|} \left( G'_{h,r}(|\boldsymbol{x}|) + \frac{1}{|\boldsymbol{x}|} \right),
\end{equation*}
and thus 
\begin{align*}
\int_{\mathbb{R}^2} \left| \boldsymbol{K}^\varepsilon(\boldsymbol{x}) - \boldsymbol{K}(\boldsymbol{x}) \right| d\boldsymbol{x} &= \frac{\varepsilon}{2 \pi} \int_{\mathbb{R}^2} \left| G'_{h,r}(|\boldsymbol{x}|) + \frac{1}{|\boldsymbol{x}|} \right| d\boldsymbol{x} \\
&= \varepsilon \int_0^\infty \left| r G'_{h,r}(r) + 1 \right| dr.
\end{align*}
Since $G_h$ is a solution to the Poisson equation, 
\begin{equation*}
- \Delta G_h(\boldsymbol{x}) = - \frac{1}{r}\frac{d}{dr}\left( r G'_{h,r}(r) \right) = h_r(r), 
\end{equation*}
and (\ref{G'_h-lim}) gives $r G'_{h,r}(r) \rightarrow - 1$ as $r \rightarrow \infty$, we obtain 
\begin{equation*}
1 + r G'_{h,r}(r) = \int_r^\infty k h_r(k) dk.
\end{equation*}
Hence, it follows that
\begin{align*}
\int_0^\infty \left| r G'_{h,r}(r) + 1 \right| dr &\leq  \int_0^\infty \int_r^\infty k \left| h_r(k) \right| dk  dr \\
& = \left[ r \int_r^\infty k \left| h_r(k) \right| dk \right]_0^\infty + \int_0^\infty r^2 \left| h_r(r) \right| dr \\
& = \int_0^\infty r^2 \left| h_r(r) \right| dr \\
& = \frac{1}{2 \pi} \| \chi_1 h \|_{L^1}.
\end{align*}
Here, we have used the fact that $h \in L^1(\mathbb{R}^2)$ and $\chi_1 h \in L^1(\mathbb{R}^2)$ yield
\begin{equation*}
\left[ r \int_r^\infty k \left| h_r(k) \right| dk \right]_0^\infty = 0.
\end{equation*}
The proof of the proposition is now complete.
\end{proof}

\section{Proofs of main theorems}
\subsection{Global solvability in the space of Radon measure}
We prove Theorem~\ref{well-posed} by a classical iterative method. In this method, a uniform boundedness and a quasi-Lipschitz continuity of the velocity field play an essential role. Actually, the following lemma holds.

\begin{lem}
Consider the initial value problem for $\boldsymbol{x}(t) \in \mathbb{R}^n$, $t \in [0, \infty)$.
\begin{equation} 
\frac{d}{dt} \boldsymbol{x}(t) = \boldsymbol{b} (\boldsymbol{x}(t), t), \qquad \boldsymbol{x}(0) = \boldsymbol{x}_0,  \label{ode}
\end{equation}
where $\boldsymbol{b} \in C\left( \mathbb{R}^n \times [0,\infty); \mathbb{R}^n \right)$ is uniformly bounded and quasi-Lipschitz continuous in space, namely
\begin{equation*}
\left| \boldsymbol{b}(\boldsymbol{x}, t) - \boldsymbol{b}(\boldsymbol{x}', t) \right| \leq c\varphi\left( |\boldsymbol{x} - \boldsymbol{x}'| \right),
\end{equation*}
and $c$ is independent of $t$. Then, (\ref{ode}) has a global unique solution $\boldsymbol{x}(\cdot) \in C^1[0, \infty)$.
\label{ODE}
\end{lem}
The proof of the lemma is given in \cite{Marchioro}. Our proof of Theorem~\ref{well-posed} is also based on that in \cite{Marchioro, Oliver}. We begin with introducing a sequence of approximate solutions for $n \in \mathbb{N}$, which are given by
\begin{align*}
q^0(\boldsymbol{x}, t) &= q_0(\boldsymbol{x}), \\ 
\boldsymbol{u}^n(\boldsymbol{x}, t)& = \left( \boldsymbol{K}_h \ast q^{n-1} \right)(\boldsymbol{x}, t),  \\ 
\partial_t\boldsymbol{\eta}^n(\boldsymbol{x},t) &= \boldsymbol{u}^n(\boldsymbol{\eta}^n(\boldsymbol{x},t), t),  \\ 
\boldsymbol{\eta}^n(\boldsymbol{x},0) &= \boldsymbol{x},  \\
q^n(\boldsymbol{x}, t) &= q_0(\boldsymbol{\eta}^n(\boldsymbol{x},-t) ).
\end{align*}
As the first step, we prove that $\boldsymbol{\eta}^n \in C^1([0,\infty) ; \mathscr{G})$ for $n \in \mathbb{N}$ by induction. For $n = 1$, it follows from (\ref{unif-bdd}) and Lemma~\ref{Kh-est} that $\| \boldsymbol{u}^1 \|_{L^\infty} \leq \| \boldsymbol{K}_h \|_{L^\infty} \| q_0 \|_{\mathcal{M}}$ and
\begin{align*}
\left| \boldsymbol{u}^1(\boldsymbol{x}) - \boldsymbol{u}^1(\boldsymbol{x}') \right| &\leq \sup_{\boldsymbol{y} \in \mathbb{R}^2} \left| \boldsymbol{K}_h(\boldsymbol{x} - \boldsymbol{y}) - \boldsymbol{K}_h(\boldsymbol{x}' - \boldsymbol{y}) \right| \| q_0 \|_{\mathcal{M}}  \\
& \leq c \varphi\left( |\boldsymbol{x} - \boldsymbol{x}'| \right) \| q_0 \|_{\mathcal{M}}.
\end{align*}
Thus, $\boldsymbol{u}^1$ is uniformly bounded on $\mathbb{R}^2 \times[0,\infty)$, quasi-Lipschitz in space and obviously continuous in time. Using Lemma~\ref{ODE}, we obtain $\boldsymbol{\eta}^1 \in C^1([0,\infty) ; \mathscr{G})$. For $\boldsymbol{\eta}^{n-1} \in C^1([0,\infty) ; \mathscr{G})$, considering $\| q^{n-1}(t) \|_{\mathcal{M}} = \| q_0 \|_{\mathcal{M}}$, we readily find that $\boldsymbol{u}^n$ is uniformly bounded and quasi-Lipschitz in space as well. Regarding the continuity in time, we have
\begin{align*}
|\boldsymbol{u}^n(\boldsymbol{x}, t) - \boldsymbol{u}^n(\boldsymbol{x}, t')| & \leq \sup_{\boldsymbol{y} \in \mathbb{R}^2} \left| \boldsymbol{K}_h(\boldsymbol{x} - \boldsymbol{\eta}^{n-1}(\boldsymbol{y},t)) - \boldsymbol{K}_h(\boldsymbol{x} - \boldsymbol{\eta}^{n-1}(\boldsymbol{y},t'))\right| \| q_0 \|_{\mathcal{M}} \\
& \leq c \sup_{\boldsymbol{y} \in \mathbb{R}^2} \varphi \left( \left| \boldsymbol{\eta}^{n-1}(\boldsymbol{y},t) - \boldsymbol{\eta}^{n-1}(\boldsymbol{y},t') \right| \right) \| q_0 \|_{\mathcal{M}} \\
& \leq c \sup_{\boldsymbol{y} \in \mathbb{R}^2} \sup_{s \in [t, t']} \varphi \left( \left| \partial_t\boldsymbol{\eta}^{n-1}(\boldsymbol{y},s) \right| |t - t'| \right) \| q_0 \|_{\mathcal{M}} \\
& \leq c \sup_{\boldsymbol{y} \in \mathbb{R}^2} \sup_{s \in [t, t']} \varphi \left( \left| \boldsymbol{u}^{n-1}(\boldsymbol{y},s) \right| |t - t'| \right) \| q_0 \|_{\mathcal{M}}. 
\end{align*}
Since $\boldsymbol{u}^{n}$ is uniformly bounded, a uniform continuity in time holds. Hence, it follows from Lemma~\ref{ODE} that $\boldsymbol{\eta}^n \in C^1([0,\infty) ; \mathscr{G})$.

Next, we show the existence of a limit flow $\overline{\boldsymbol{\eta}} \in C^1([0,\infty) ; \mathscr{G})$. The sequence $\left\{ \boldsymbol{\eta}^n \right\}_{n \in \mathbb{N}}$ is Cauchy in $C([0,T) ; \mathscr{G})$ for some $T > 0$. Indeed, for $n \geq N$, we have
\begin{align*}
&\left| \boldsymbol{u}^n \left( \boldsymbol{\eta}^n(\boldsymbol{x}) \right) - \boldsymbol{u}^{n-1} \left( \boldsymbol{\eta}^{n-1}(\boldsymbol{x}) \right) \right|  \\
&\leq \int_{\mathbb{R}^2} \left| \boldsymbol{K}_h \left( \boldsymbol{\eta}^n(\boldsymbol{x}) - \boldsymbol{\eta}^{n-1}(\boldsymbol{y}) \right) - \boldsymbol{K}_h \left( \boldsymbol{\eta}^{n-1}(\boldsymbol{x}) - \boldsymbol{\eta}^{n-1}(\boldsymbol{y}) \right) \right| |q_0(\boldsymbol{y})| d\boldsymbol{y} \\
& \quad + \int_{\mathbb{R}^2} \left| \boldsymbol{K}_h \left( \boldsymbol{\eta}^{n-1}(\boldsymbol{x}) - \boldsymbol{\eta}^{n-1}(\boldsymbol{y}) \right) - \boldsymbol{K}_h \left( \boldsymbol{\eta}^{n-1}(\boldsymbol{x}) - \boldsymbol{\eta}^{n-2}(\boldsymbol{y}) \right) \right| |q_0(\boldsymbol{y})| d\boldsymbol{y} \\
&\leq c \| q_0 \|_{\mathcal{M}} \varphi\left( \left| \boldsymbol{\eta}^n(\boldsymbol{x}) - \boldsymbol{\eta}^{n-1}(\boldsymbol{x}) \right| \right) + c \int_{\mathbb{R}^2} \varphi\left( \left| \boldsymbol{\eta}^{n-1}(\boldsymbol{y}) - \boldsymbol{\eta}^{n-2}(\boldsymbol{y}) \right| \right) | q_0(\boldsymbol{y}) | d\boldsymbol{y} \\
&\leq c \| q_0 \|_{\mathcal{M}} \varphi\left( \sup_{n \geq N-1} \sup_{\boldsymbol{x} \in \mathbb{R}^2} \left| \boldsymbol{\eta}^n(\boldsymbol{x}) - \boldsymbol{\eta}^{n-1}(\boldsymbol{x}) \right| \right),
\end{align*}
in which we drop the explicit time dependence. Defining 
\begin{equation*}
\rho^N(t) \equiv \sup_{n \geq N} \sup_{\boldsymbol{x} \in \mathbb{R}^2} \left| \boldsymbol{\eta}^n(\boldsymbol{x}, t) - \boldsymbol{\eta}^{n-1}(\boldsymbol{x}, t) \right|,
\end{equation*}
we obtain
\begin{align*}
\rho^N(t) &\leq \sup_{n \geq N} \sup_{\boldsymbol{x} \in \mathbb{R}^2}\int_0^t \left| \boldsymbol{u}^n \left( \boldsymbol{\eta}^n(\boldsymbol{x}, s),s  \right) - \boldsymbol{u}^{n-1} \left( \boldsymbol{\eta}^{n-1}(\boldsymbol{x}, s),s  \right) \right| ds \\  
&\leq c \| q_0 \|_{\mathcal{M}} \int_0^t \varphi\left( \rho^{N-1}(s) \right) ds.
\end{align*}
This implies $\rho^N \rightarrow 0$ uniformly on $[0, T]$ as $N \rightarrow + \infty$ for sufficiently small $T$. Since $T$ depends only on $\| q_0 \|_{\mathcal{M}}$, the convergence holds for arbitrary times and thus the limit flow $\overline{\boldsymbol{\eta}} \in C([0,\infty) ; \mathscr{G})$ exits. 

In turn, we see the convergences of the vorticity and the velocity. Let us define $\overline{q}$ and $\overline{\boldsymbol{u}}$ by
\begin{equation}
\overline{q}(\boldsymbol{x}, t) = q_0\left( \overline{\boldsymbol{\eta}}(\boldsymbol{x}, -t) \right), \qquad \overline{\boldsymbol{u}}(\boldsymbol{x}, t) = \int_{\mathbb{R}^2} \boldsymbol{K}_h(\boldsymbol{x} - \boldsymbol{y}) \overline{q}(\boldsymbol{y}, t) d\boldsymbol{y}. \label{limit_q-u}
\end{equation}
Then, we find that $q^n \rightharpoonup \overline{q}$ weakly in $\mathcal{M}(\mathbb{R}^2)$ and $\boldsymbol{u}^n \rightarrow \overline{\boldsymbol{u}}$ in $C(\mathbb{R}^2)$. Note that both converge uniformly on any finite time interval. Indeed, for any $\psi \in C_0^1(\mathbb{R}^2)$, it follows that
\begin{align*}
\left| \int_{\mathbb{R}^2} \left( q^n(\boldsymbol{x}, t) - \overline{q}(\boldsymbol{x}, t) \right) \psi(\boldsymbol{x}) d\boldsymbol{x} \right| &\leq \int_{\mathbb{R}^2} |q_0(\boldsymbol{x})| \left| \psi\left( \boldsymbol{\eta}^n(\boldsymbol{x}, t) \right) - \psi\left( \overline{\boldsymbol{\eta}}(\boldsymbol{x}, t) \right) \right| d\boldsymbol{x} \\
& \leq \| q_0 \|_{\mathcal{M}} \| \nabla \psi \|_{L^\infty} \sup_{\boldsymbol{x} \in \mathbb{R}^2} \left| \boldsymbol{\eta}^n(\boldsymbol{x}, t) - \overline{\boldsymbol{\eta}}(\boldsymbol{x}, t) \right|, 
\end{align*}
and  
\begin{align*}
\left| \boldsymbol{u}^n(\boldsymbol{x},t) - \overline{\boldsymbol{u}}(\boldsymbol{x},t) \right| & \leq c \| q_0 \|_{\mathcal{M}} \varphi\left( \sup_{\boldsymbol{y} \in \mathbb{R}^2} \left| \boldsymbol{\eta}^n(\boldsymbol{y}, t) - \overline{\boldsymbol{\eta}}(\boldsymbol{y}, t) \right| \right).
\end{align*}
Hence, the convergence of $\boldsymbol{\eta}^n$ and the density argument yield the above assertions. 

Finally, we show that $(\overline{\boldsymbol{\eta}}, \overline{q}, \overline{\boldsymbol{u}})$ is the solution to (\ref{h-eta}). Consider the integral form of (\ref{h-eta}), 
\begin{align*}
& \left| \overline{\boldsymbol{\eta}}(\boldsymbol{x}, t) - \boldsymbol{x} - \int_0^t \overline{\boldsymbol{u}}\left( \overline{\boldsymbol{\eta}}(\boldsymbol{x}, s), s \right) ds  \right| \\
& = \left| \overline{\boldsymbol{\eta}}(\boldsymbol{x}, t) - \boldsymbol{x} - \int_0^t \overline{\boldsymbol{u}}\left( \overline{\boldsymbol{\eta}}(\boldsymbol{x}, s), s \right) ds - \left( \boldsymbol{\eta}^n(\boldsymbol{x}, t) - \boldsymbol{x} - \int_0^t \boldsymbol{u}^n\left( \boldsymbol{\eta}^n(\boldsymbol{x}, s), s \right) ds \right)  \right| \\
& \leq \left| \boldsymbol{\eta}^n(\boldsymbol{x}, t) - \overline{\boldsymbol{\eta}}(\boldsymbol{x}, t) \right| + \int_0^t \left| \boldsymbol{u}^n\left( \boldsymbol{\eta}^n(\boldsymbol{x}, s), s \right) - \overline{\boldsymbol{u}}\left( \boldsymbol{\eta}^n(\boldsymbol{x}, s), s \right) \right| ds \\
& \hspace{40mm} + \int_0^t \left| \overline{\boldsymbol{u}}\left( \boldsymbol{\eta}^n(\boldsymbol{x}, s), s \right) - \overline{\boldsymbol{u}}\left( \overline{\boldsymbol{\eta}}(\boldsymbol{x}, s), s \right) \right| ds \\
& \leq \sup_{\boldsymbol{x} \in \mathbb{R}^2} \left| \boldsymbol{\eta}^n(\boldsymbol{x}, t) - \overline{\boldsymbol{\eta}}(\boldsymbol{x}, t) \right| + c \| q_0 \|_{\mathcal{M}} \int_0^t \varphi\left( \sup_{\boldsymbol{y} \in \mathbb{R}^2} \left| \boldsymbol{\eta}^{n-1}(\boldsymbol{y}, s) - \overline{\boldsymbol{\eta}}(\boldsymbol{y}, s) \right| \right) ds \\
& \hspace{48mm} + c \| q_0 \|_{\mathcal{M}} \int_0^t \varphi\left( \sup_{\boldsymbol{x} \in \mathbb{R}^2} \left| \boldsymbol{\eta}^n(\boldsymbol{x}, s) - \overline{\boldsymbol{\eta}}(\boldsymbol{x}, s) \right| \right) ds.
\end{align*}
Then, the right side converges zero uniformly in space as $n \rightarrow + \infty$. Since $\overline{\boldsymbol{u}}\left( \overline{\boldsymbol{\eta}}(\boldsymbol{x}, t), t \right)$ is continuous in time, $\overline{\boldsymbol{\eta}}$ satisfies (\ref{h-eta}) and $\partial_t \overline{\boldsymbol{\eta}}$ is continuous in time. Considering the time reversibility of (\ref{h-eta}), we find that the solutions are extended to negative times. Moreover, we can see that the solutions $\overline{q}$ and $\overline{\boldsymbol{u}}$ satisfy
\begin{equation*}
\int_{\mathbb{R}} \int_{\mathbb{R}^2} \left( \partial_t \psi + \overline{\boldsymbol{u}} \cdot \nabla \psi \right) \overline{q} d\boldsymbol{x} dt = 0,
\end{equation*}
for any $\psi \in C_0^\infty(\mathbb{R}\times\mathbb{R}^2)$. The uniqueness of the solutions can be proven based on the same estimates in the existence part.

\label{proof1}

\subsection{Convergence to the Euler equations}
In this section, we prove Theorem~\ref{convergence}. The proof proceeds in the same way as \cite{Oliver}. Before that, let us review the following quasi-Lipschitz estimates for the singular kernel $\boldsymbol{K}$.
\begin{lem}
For $\omega \in L^1(\mathbb{R}^2) \cap L^\infty(\mathbb{R}^2)$, 
\begin{equation*}
\int_{\mathbb{R}^2} \left| \boldsymbol{K}(\boldsymbol{x} - \boldsymbol{y}) - \boldsymbol{K}(\boldsymbol{x}' - \boldsymbol{y}) \right| \left| \omega(\boldsymbol{y}) \right| d\boldsymbol{y} \leq c \varphi\left( \left| \boldsymbol{x} - \boldsymbol{x}' \right| \right) \left( \|\omega\|_{L^1} + \|\omega\|_{L^\infty} \right).
\end{equation*}
\label{K-esti1}
\end{lem}
The proof of this lemma is straightforward and we find it, for example, in \cite{Benedetto, Marchioro, McGrath}. The following lemma is a modified estimate of the singular kernel, which was proven in \cite{Benedetto, Oliver}.
\begin{lem}
Let $\omega \in L^1(\mathbb{R}^2) \cap L^\infty(\mathbb{R}^2)$ and $\phi$ be an area preserving measurable transformation on $\mathbb{R}^2$. Then 
\begin{equation*}
\int_{\mathbb{R}^2} \left| \boldsymbol{K}(\boldsymbol{x} - \boldsymbol{y}) - \boldsymbol{K}(\boldsymbol{x} - \phi( \boldsymbol{y} )) \right| \left| \omega(\boldsymbol{y}) \right| d\boldsymbol{y} \leq c \sup_{x \in \mathbb{R}^2} \varphi\left( \left| \boldsymbol{x} - \phi(\boldsymbol{x}) \right| \right) \left( \|\omega\|_{L^1} + \|\omega\|_{L^\infty} \right).
\end{equation*}
\label{K-esti2}
\end{lem}
We estimate the difference of the flow maps.
\begin{align*}
&\left| \boldsymbol{\eta}^\varepsilon(\boldsymbol{x}, t) - \boldsymbol{\eta}(\boldsymbol{x}, t) \right| \leq \int_0^t \left| \boldsymbol{u}^\varepsilon\left( \boldsymbol{\eta}^\varepsilon(\boldsymbol{x}, s), s \right) - \boldsymbol{u}\left( \boldsymbol{\eta}(\boldsymbol{x}, s), s \right) \right| ds  \\
&\leq \int_0^t \int_{\mathbb{R}^2} \left| \boldsymbol{K}^\varepsilon\left( \boldsymbol{\eta}^\varepsilon(\boldsymbol{x}, s) - \boldsymbol{\eta}^\varepsilon(\boldsymbol{y}, s) \right) - \boldsymbol{K} \left( \boldsymbol{\eta}^\varepsilon(\boldsymbol{x}, s) - \boldsymbol{\eta}^\varepsilon(\boldsymbol{y}, s) \right) \right| \left| \omega_0(\boldsymbol{y}) \right| d\boldsymbol{y} ds \\
&\quad + \int_0^t \int_{\mathbb{R}^2} \left| \boldsymbol{K}\left( \boldsymbol{\eta}^\varepsilon(\boldsymbol{x}, s) - \boldsymbol{\eta}^\varepsilon(\boldsymbol{y}, s) \right) - \boldsymbol{K} \left( \boldsymbol{\eta}(\boldsymbol{x}, s) - \boldsymbol{\eta}^\varepsilon(\boldsymbol{y}, s) \right) \right| \left| \omega_0(\boldsymbol{y}) \right| d\boldsymbol{y} ds \\
&\quad + \int_0^t \int_{\mathbb{R}^2} \left| \boldsymbol{K}\left( \boldsymbol{\eta}(\boldsymbol{x}, s) - \boldsymbol{\eta}^\varepsilon(\boldsymbol{y}, s) \right) - \boldsymbol{K} \left( \boldsymbol{\eta}(\boldsymbol{x}, s) - \boldsymbol{\eta}(\boldsymbol{y}, s) \right) \right| \left| \omega_0(\boldsymbol{y}) \right|  d\boldsymbol{y} ds \\
&\equiv \int_0^t J_1 + J_2 + J_3 ds.
\end{align*}
As for $J_1$, it follows from Lemma~\ref{K_con} that
\begin{equation*}
J_1 \leq \| \omega_0 \|_{L^\infty} \int_{\mathbb{R}^2} \left| \boldsymbol{K}^\varepsilon(\boldsymbol{y}) - \boldsymbol{K}(\boldsymbol{y}) \right| d\boldsymbol{y} \leq \frac{\varepsilon}{2 \pi} \| \omega_0 \|_{L^\infty} \| \chi_1 h \|_{L^1} .
\end{equation*}
The other two terms are estimated by using Lemma~\ref{K-esti1} and Lemma~\ref{K-esti2} respectively. We find
\begin{align*}
J_2 &= \int_{\mathbb{R}^2}  \left| \boldsymbol{K}\left( \boldsymbol{\eta}^\varepsilon(\boldsymbol{x}, s) - \boldsymbol{y} \right) - \boldsymbol{K} \left( \boldsymbol{\eta}(\boldsymbol{x}, s) - \boldsymbol{y} \right)  \right| \left| q(\boldsymbol{y}, s) \right|  d\boldsymbol{y} \\
&\leq c_2 \varphi \left( \left| \boldsymbol{\eta}^\varepsilon(\boldsymbol{x}, s) - \boldsymbol{\eta}(\boldsymbol{x}, s) \right| \right) \left( \| \omega_0 \|_{L^1} + \| \omega_0 \|_{L^\infty} \right),
\end{align*}
and
\begin{align*}
J_3 &= \int_{\mathbb{R}^2} \left| \boldsymbol{K}\left( \boldsymbol{\eta}(\boldsymbol{x}, s) - \boldsymbol{\eta}^\varepsilon \left( \boldsymbol{\eta}^{-1}(\boldsymbol{y}, s), s \right) \right) - \boldsymbol{K} \left( \boldsymbol{\eta}(\boldsymbol{x}, s) - \boldsymbol{y} \right) \right| \left| \omega(\boldsymbol{y}, s) \right|  d\boldsymbol{y} \\
&\leq c_3 \sup_{\boldsymbol{x}\in\mathbb{R}^2} \varphi \left( \left| \boldsymbol{\eta}^\varepsilon(\boldsymbol{x}, s)  - \boldsymbol{\eta}(\boldsymbol{x}, s)  \right| \right) \left( \| \omega_0 \|_{L^1} + \| \omega_0 \|_{L^\infty} \right).
\end{align*}
Combining above estimates and setting $\rho(t) = \sup_{\boldsymbol{x}\in\mathbb{R}^2} \left| \boldsymbol{\eta}^\varepsilon(\boldsymbol{x}, t)  - \boldsymbol{\eta}(\boldsymbol{x}, t) \right|$, we obtain
\begin{equation*}
\rho(t) \leq c \int_0^t \left[ \varepsilon + \varphi\left( \rho(s) \right) \right] ds.
\end{equation*}
This yields $\rho(t) \leq C(T) \varepsilon^{\exp{(-T)}}$ for $t\in [0, T]$, see the calculation in \cite{Oliver}. The proof is now complete.

\label{proof2}


\end{document}